\documentclass{ifacconf}

\usepackage{graphicx}      % include this line if your document contains figures
\usepackage{natbib}        % required for bibliography
\usepackage{amsmath}
\usepackage{amsfonts}
\usepackage{mathtools}
\usepackage{MnSymbol}
\usepackage{color}
\usepackage{float}
\ifdefined\beginlemma\else
\newtheorem{lemma}{Lemma}
\ifdefined\beginproposition\else
\newtheorem{proposition}{Proposition}
\ifdefined\begintheorem\else

\usepackage{lineno}
\makeatletter
\let\old@ssect\@ssect % Store how ifacconf defines \@ssect
\makeatother

\usepackage{natbib}
\usepackage{hyperref}

\makeatletter
\def\@ssect#1#2#3#4#5#6{%
  \NR@gettitle{#6}% Insert key \nameref title grab
  \old@ssect{#1}{#2}{#3}{#4}{#5}{#6}% Restore ifacconf's \@ssect
}
\makeatother
\graphicspath{{images/}}
\newcommand{\grad}{\normalfont{\text{grad}}}

\newcommand{\R}[1][\empty]{\mathbb{R}^{#1}}

\makeatletter
\DeclareRobustCommand{\qed}{%
  \ifmmode % if math mode, assume display: omit penalty etc.
  \else \leavevmode\unskip\penalty9999 \hbox{}\nobreak\hfill
  \fi
  \quad\hbox{\qedsymbol}}
\newcommand{\openbox}{\leavevmode
  \hbox to.77778em{%
  \hfil\vrule
  \vbox to.675em{\hrule width.6em\vfil\hrule}%
  \vrule\hfil}}
\newcommand{\qedsymbol}{\openbox}
\newenvironment{proof}[1][\proofname]{\par
  \normalfont
  \topsep6\p@\@plus6\p@ \trivlist
  \item[\hskip\labelsep\itshape
    #1.]\ignorespaces
}{%
  \qed\endtrivlist
}
\newcommand{\proofname}{Proof}
\makeatother

\newcommand{\so}{\mathfrak{so}}

\newcommand{\g}{\mathfrak{g}}

\newcommand{\Exp}{\normalfont{\text{Exp}}}
\newcommand{\Log}{\normalfont{\text{Log}}}
\newcommand{\ext}{\normalfont{\text{ext}}}
\newcommand{\Bi}{\normalfont{\text{Bi}}}
\newcommand{\N}{\mathbb{N}}

\newcommand{\tr}{\normalfont{\text{tr}}}
\newcommand{\SO}{\normalfont{\text{SO}}}

\newcommand{\ad}{\normalfont{\text{ad}}}
\newcommand{\Ad}{\normalfont{\text{Ad}}}

\newcommand{\coll}{\normalfont{\text{coll}}}

\makeatletter
\def\widebreve{\mathpalette\wide@breve}
\def\wide@breve#1#2{\sbox\z@{$#1#2$}%
     \mathop{\vbox{\m@th\ialign{##\crcr
\kern0.08em\brevefill#1{0.8\wd\z@}\crcr\noalign{\nointerlineskip}%
                    $\hss#1#2\hss$\crcr}}}\limits}
\def\brevefill#1#2{$\m@th\sbox\tw@{$#1($}%
  \hss\resizebox{#2}{\wd\tw@}{\rotatebox[origin=c]{90}{\upshape(}}\hss$}
\makeatletter

\DeclarePairedDelimiterX\set[1]\lbrace\rbrace{#1}

\newtheorem{corollary}{Corollary}

\theoremstyle{definition}

\newtheorem{remark}{Remark}
\newtheorem{example}{Example}

\begin{document}
\begin{frontmatter}

\title{Reduction of Necessary Conditions for the Variational Collision Avoidance Problem\thanksref{footnoteinfo}} 

\thanks[footnoteinfo]{The authors acknowledge financial support from Grant PID2022-137909NB-C21 funded by MCIN/AEI/ 10.13039/501100011033.}

\author[First]{Jacob R. Goodman} 
\author[Second]{Leonardo J. Colombo} 

\address[First]{J. Goodman is with Antonio de Nebrija University, Departamento de Informática, Escuela Politécnica Superior, C. de Sta. Cruz de Marcenado, 27, 28015, Madrid, Spain. email: jacob.goodman@nebrija.es}
   \address[Second]{L.Colombo is with Centre for Automation and Robotics (CSIC-UPM), Ctra. M300 Campo Real, Km 0,200, Arganda
del Rey - 28500 Madrid, Spain. email:  leonardo.colombo@csic.es}

\begin{abstract}                % Abstract of 50--100 words
In this work we study the reduction by a Lie group of symmetries of variational collision avoidance probelms of multiple agents evolving on a Riemannian manifold and derive necessary conditions for the reduced extremals. The problem consists of finding non-intersecting trajectories of a given number of agents, among a set of admissible curves, to reach a specified configuration, based on minimizing an energy functional that depends on the velocity, covariant acceleration and an artificial potential function used to prevent collision among the agents.
\end{abstract}


\end{frontmatter}
%===============================================================================

\section{Introduction}
Dimensionality reduction for large scale systems has become an active problem of interest within the automatic control and robotics communities. In multi-agent systems, guidance and trajectory planning algorithms for coordination while optimizing qualitative features for the system of multiple robots are determined by solutions of nonlinear equations which demand a high-computational costs along its integration. The construction of methods for the reduction of dimensionality permits fast computations for the generation of optimal trajectories in the collision avoidance motion of multi-agent systems.%The construction of control laws on vector space permits to employ fast integration algorithms for collision avoidance in.
%and  to reduce computational costs,

Methods for trajectory tracking and estimation algorithms for pose and attitude of mechanical systems evolving on Lie groups are commonly employed for improving accuracy on simulations, as well as to avoid singularities by working with coordinate-free expressions  in the associated Lie algebra of the Lie group to describe behaviors in multi-agent systems (i.e., a set of equations depending on an arbitrary choice of the basis for the Lie algebra). More recently, this framework has been used for cooperative transportation \cite{goodman2022geometric}, \cite{goodman2023geometric}. 
%coordinate-free expressions for the dynamics describing the behavior of the system (i.e., only depends on an arbitrary choice of the basis for the Lie algebra)

Optimization problems on Lie groups have a long history \cite{jurdjevic1997geometric} and have been applied to many problems in control engineering. In practice, many robotic systems exhibit symmetries that can be exploited to reduce some of the complexities of system
models, for instance degrees of freedom. Symmetries in optimal control for systems on Lie groups have been studied in \cite{Bl}, \cite{krishnaprasad1993optimal}, \cite{koon1997optimal}, \cite{grizzle1985structure} among many others, mainly for applications in robotic and aerospace engineering, and in particular, for spacecraft attitude control and underwater vehicles \cite{leonard1995motion}. While most of the applications of symmetry reduction provided in the literature focus on the single agent situation, only a few works studied the relation between multi-agent systems and symmetry reduction (see for instance the early work on the topic \cite{justh2004equilibria}), in this work we consider symmetry reduction of multi-agent systems in the necessary conditions for optimality obtained via a variational problem on Lie groups, with a decentralized communication topology determined by an undirected graph, i.e., the information between the agents is only shared between nearest neighbors.

Riemannian polynomials are smooth and optimal in the sense that they minimize the average square magnitude of some higher-order derivative along the curve. This quantity is often related to the magnitude of the controller in control engineering applications (which itself is related to energy consumption). Moreover, Riemannian polynomials carry a rich geometry with them, which has been studied extensively in the literature (see \cite{Giambo, marg, noakes} for a detailed account of Riemannian cubics and \cite{RiemannianPoly} for some results with higher-order Riemannian polynomials).

It is often the case that—in addition to interpolating points—there are obstacles or regions in space that need to be avoided. In this case, a typical strategy is to augment the action functional with an artificial potential term that grows large near the obstacles and small away from them (in that sense, the trajectories that minimize the action are expected to avoid the obstacles). This was done for instance in \cite{BlCaCoCDC}, \cite{BlCaCoIJC}, \cite{CoGo20}, \cite{colombo2023existence} where necessary conditions for extrema in obstacle avoidance problems on Riemannian manifolds were derived. In addition to applications to interpolation problems on manifolds and to energy-minimum problems on Lie groups and symmetric spaces endowed with a bi-invariant metric \cite{point}, and extended in \cite{mishal}, \cite{sh} and \cite{goodman2022collision} for the collision avoidance task and hybrid systems in \cite{goodman2021obstacle}. Reduction of necessary conditions for the obstacle avoidance problem were studied in \cite{goodman2022reduction} and sufficient conditions for the problem were studied in \cite{goodman2022sufficient}. In this paper, we build on the previous studies by considering the problem of reduction by a Lie group of symmetries necessary conditions for optimality in the variational collision avoidance problem on Lie groups endowed with a left-invariant metric. Finally, a brief study of the reduction by symmetries of the collision avoidance problem in the case of bi-invariant metrics is considered.

\section{Background on Riemannian manifolds and Global analysis}\label{Sec: background}
\subsection{Background on Riemannian manifolds}
Let $(Q, \left< \cdot, \cdot\right>)$ be an $n$-dimensional \textit{Riemannian
manifold}, where $Q$ is an $n$-dimensional smooth manifold and $\left< \cdot, \cdot \right>$ is a positive-definite symmetric covariant 2-tensor field called the \textit{Riemannian metric}. That is, to each point $q\in Q$ we assign a positive-definite inner product $\left<\cdot, \cdot\right>_q:T_qQ\times T_qQ\to\mathbb{R}$, where $T_qQ$ is the \textit{tangent space} of $Q$ at $q$ and $\left<\cdot, \cdot\right>_q$ varies smoothly with respect to $q$. The length of a tangent vector is determined by its norm, defined by
$\|v_q\|=\left<v_q,v_q\right>^{1/2}$ with $v_q\in T_qQ$. For any $p \in Q$, the Riemannian metric induces an invertible map $\cdot^\flat: T_p Q \to T_p^\ast Q$, called the \textit{flat map}, defined by $X^\flat(Y) = \left<X, Y\right>$ for all $X, Y \in T_p Q$. The inverse map $\cdot^\sharp: T_p^\ast Q \to T_p Q$, called the \textit{sharp map}, is similarly defined implicitly by the relation $\left<\alpha^\sharp, Y\right> = \alpha(Y)$ for all $\alpha \in T_p^\ast Q$. Let $C^{\infty}(Q)$ and $\Gamma(TQ)$ denote the spaces of smooth scalar fields and smooth vector fields on $Q$, respectively. The sharp map provides a map from $C^{\infty}(Q) \to \Gamma(TQ)$ via $\grad f(p) = df_p^\sharp$ for all $p \in Q$, where $\grad f$ is called the \textit{gradient vector field} of $f \in C^{\infty}(Q)$. More generally, given a map $V: Q \times \cdots \times Q \to \R$ (with $m$ copies of $Q$), we may consider the gradient vector field of $V$ with respect to $i^{\text{th}}$ component as $\grad_i V(q_1, \dots, q_m) = \grad U(q_i)$, where $U(q) = V(q_1, \dots, q_{i-1}, q, q_{i+1}, \dots, q_m)$ for all $q, q_1, \dots, q_m \in Q$.

Vector fields are a special case of smooth sections of vector bundles. In particular, given a vector bundle $(E, Q, \pi)$ with total space $E$, base space $Q$, and projection $\pi: E \to Q$, where $E$ and $Q$ are smooth manifolds, a \textit{smooth section} is a smooth map $X: Q \to E$ such that $\pi \circ X = \text{id}_Q$, the identity function on $Q$. We similarly denote the space of smooth sections on $(E, Q, \pi)$ by $\Gamma(E)$. A \textit{connection} on $(E, Q, \pi)$ is a map $\nabla: \Gamma(TQ) \times \Gamma(E) \to \Gamma(TQ)$ which is $C^{\infty}(Q)$-linear in the first argument, $\R$-linear in the second argument, and satisfies the product rule $\nabla_X (fY) = X(f) Y + f \nabla_X Y$ for all $f \in C^{\infty}(Q), \ X \in \Gamma(TQ), \ Y \in \Gamma(E)$. The connection plays a role similar to that of the directional derivative in classical real analysis. The operator
$\nabla_{X}$ which assigns to every smooth section $Y$ the vector field $\nabla_{X}Y$ is called the \textit{covariant derivative} (of $Y$) \textit{with respect to $X$}.

Connections induces a number of important structures on $Q$, a particularly ubiquitous such structure is the \textit{curvature endomorphism}, which is a map $R: \Gamma(TQ) \times \Gamma(TQ) \times \Gamma(E) \to \Gamma(TQ)$ defined by $R(X,Y)Z := \nabla_{X}\nabla_{Y}Z-\nabla_{Y}\nabla_{X}Z-\nabla_{[X,Y]}Z$ for all $X, Y \in \Gamma(TQ), \ Z \in \Gamma(E)$.  The curvature endomorphism measures the extent to which covariant derivatives commute with one another. %We further define the \textit{curvature tensor} $\text{Rm}$ on $Q$ via $\text{Rm}(X, Y, Z, W) := \left<R(X, Y)Z, W\right>$.

We now specialize our attention to \textit{affine connections}, which are connections on $TQ$. Let $q: I \to Q$ be a smooth curve parameterized by $t \in I \subset \R$, and denote the set of smooth vector fields along $q$ by $\Gamma(q)$. Then for any affine connection $\nabla$ on $Q$, there exists a unique operator $D_t: \Gamma(q) \to \Gamma(q)$ (called the \textit{covariant derivative along $q$}) which agrees with the covariant derivative $\nabla_{\dot{q}}\tilde{W}$ for any extension $\tilde{W}$ of $W$ to $Q$. A vector field $X \in \Gamma(q)$ is said to be \textit{parallel along $q$} if $\displaystyle{D_t X\equiv 0}$. %For $k\in \N$, the $k$th-order covariant derivative of $W$ along $q$, denoted by $\displaystyle{D_t^k W}$, can then be inductively defined by $\displaystyle{D_t^k W = D_t \left(D_t^{k-1} W\right)}$. 

The covariant derivative allows to define a particularly important family of smooth curves on $Q$ called \textit{geodesics}, which are defined as the smooth curves $\gamma$ satisfying $D_t \dot{\gamma} = 0$. Moreover, geodesics induce a map $\mathrm{exp}_q:T_qQ\to Q$ called the \textit{exponential map} defined by $\mathrm{exp}_q(v) = \gamma(1)$, where $\gamma$ is the unique geodesic verifying $\gamma(0) = q$ and $\dot{\gamma}(0) = v$. In particular, $\mathrm{exp}_q$ is a diffeomorphism from some star-shaped neighborhood of $0 \in T_q Q$ to a convex open neighborhood $\mathcal{B}$ (called a \textit{goedesically convex neighborhood}) of $q \in Q$. It is well-known that the Riemannian metric induces a unique torsion-free and metric compatible connection called the \textit{Riemannian connection}, or the \textit{Levi-Civita connection}. Along the remainder of this paper, we will assume that $\nabla$ is the Riemannian connection. For additional information on connections and curvature, we refer the reader to \cite{Boothby}. When the covariant derivative $D_t$ corresponds to the Levi-Civita connection, geodesics can also be characterized as the critical points of the length functional $\displaystyle{L(\gamma) = \int_0^1 \|\dot{\gamma}\|dt}$ among all unit-speed \textit{piece-wise regular} curves $\gamma: [a, b] \to Q$ (that is, where there exists a subdivision of $[a, b]$ such that $\gamma$ is smooth and satisfies $\dot{\gamma} \ne 0$ on each subdivision). %Equivalently, we may characterize geodesics by the critical points of the \textit{energy functional} $\displaystyle{\mathcal{E} = \frac12 \int_a^b \|\dot{\gamma}\|^2 dt}$ among all $C^1$ piece-wise smooth curves $\gamma: [a, b] \to Q$ parameterized by arc-length. The length functional induces a metric $d: Q \times Q \to \R$ called the \textit{Riemannian distance} via $d(p, q) = \inf\{L(\gamma): \ \gamma \ \text{regular, } \gamma(a) = p, \ \gamma(b) = q\}$.

If we assume that $Q$ is \textit{complete} (that is, $(Q, d)$ is a complete metric space), then by the Hopf-Rinow theorem, any two points $x$ and $y$ in $Q$ can be connected by a (not necessarily unique) minimal-length geodesic $\gamma_{x,y}$. In this case, the Riemannian distance between $x$ and $y$ can be defined by 
\begin{equation}\label{eq: distance-length-geodesics}
    {d(x,y)=\int_{0}^{1}\Big{\|}\frac{d \gamma_{x,y}}{d s}(s)\Big{\|}\, ds}.
\end{equation} Moreover, if $y$ is contained in a geodesically convex neighborhood of $x$, we can write the Riemannian distance by means of the Riemannian exponential as $d(x,y)=\|\mbox{exp}_x^{-1}y\|.$

%Consider the norm $\| \cdot \|_{T_x \Omega}$ on $T_x \Omega$ given by
%$$\|X\|_{T_x \Omega} = \left(\int_a^b \left[\left\|X\right\|^2 + \left\|\frac{DX}{dt}\right\|^2 + \left\|\frac{D^2 X}{dt^2}\right\|^2 \right]dt\right)^{1/2}.$$
%We let $\mathring{H}^2_x$ be the completion of $T_x \Omega$ under $\| \cdot \|_{T_x \Omega}$. Considering an orthonormal basis of parallel vector fields $\left\{X_i\right\}$ along $x$ and writing $X = \xi^i X_i$ for some $\xi^i: [a, b] \to \R$, we have that
%$$\|X\|_{T_x \Omega} = \left(\int_a^b \left[ \xi^i \xi^i + \dot{\xi}^i \dot{\xi}^i + \ddot{\xi}^i \ddot{\xi}^i \right]dt\right)^{1/2},$$
%from which it is clear that $\mathring{H}^2_x$ can be identified with the Sobolev space $\mathring{H}^2([a,b], \R^n)$ (as discussed for instance in section 4.3 of \cite{Jost}).

%\begin{lemma}\label{lemma: commutativity second covariant derivatives}
%Suppose that $\Gamma$ is an admissible variation of $q \in \Omega^{(1)}$, and $V \in T_q \Omega^{(1)}$. Then,
%\begin{equation}\label{eq: commutativity second covariant derivatives}
 %   D_s D_t V - D_t D_s V = R(\partial_s \Gamma, \dot{\Gamma})V.
%\end{equation}
%\end{lemma}
\subsection{Sobolev Spaces of Curves}\label{Subsection: Sobolev spaces}
One often views finite-dimensional smooth manifolds as spaces which are locally diffeomorphic to $\mathbb{R}^n$ for some $n \in \N$. Infinite-dimensional manifolds are defined in much the same way, with $\mathbb{R}^n$ being replaced by some infinite-dimensional topological vector space equipped with some additional structure that allows for the notion of smoothness. Common choices include locally convex topological vector spaces, Fréchet spaces, Banach spaces, and Hilbert spaces (in decreasing order of generality), which are known as the \textit{model spaces} for the manifold. Each type of model space comes with its own advantages and disadvantages, and is often determined by the problem of interest. In this thesis, the most natural choice in model space turns out to be Hilbert spaces (particularly Sobolev spaces).

Let $I \subset \mathbb{R}$ be a closed interval and $L^2(I, \mathbb{R}^n)$ denote the space of square integrable functions $f: I \to \mathbb{R}^n$. That is, $f \in L^2(I, \mathbb{R}^n)$ if and only if $\int_I \|f(x)\|^2 dx < +\infty$, where $\|\cdot \|$ denotes the Euclidean norm on $\mathbb{R}^n$. Equivalently, we could say that $f = (f^1, \dots, f^n)$ is of class $L^2(I, \mathbb{R}^n)$ if and only if $f^i$ is of class $L^2(I, \mathbb{R})$ for all $1 \le i \le n.$ $L^2$ becomes a Hilbert space when equipped with the inner product $\left< f, g\right> = \int_I \left( f(t) \cdot g(t)\right)dt$, where where $f \cdot g$ denotes the "dot product" on $\mathbb{R}^n$.

Let $k \ge 1$ and consider functions $f, \varphi: [a, b] \to \mathbb{R}^n$ such that $\frac{d^j}{dt^j} \varphi(a) = \frac{d^j}{dt^j} \varphi(b) = 0$ for all $0 \le j \le k$. It follows via integration by parts that
\begin{equation*}
    \int_a^b \left(f(t) \cdot \frac{d^k}{dt^k} \varphi(t)\right)dt = (-1)^k \int_a^b \left(\frac{d^k}{dt^k} f(t) \cdot \varphi(t) \right)dt,
\end{equation*}
The left-hand side of the above equation still makes sense if we assume $f$ only to be integrable on $[a, b]$. If there exists an integrable function $g: [a, b] \to \mathbb{R}^n$ such that 
\begin{equation*}
    \int_a^b \left(f(t) \cdot \frac{d^k}{dt^k} \varphi(t)\right)dt = (-1)^k \int_a^b \left(g(t) \cdot \varphi(t) \right)dt
\end{equation*}
for all $\varphi:[a, b] \to \mathbb{R}^n$ which vanish at the endpoints along with its first $k$ derivatives, we refer to $g$ as the $k^{\text{th}}$ weak derivative of $f$, and often denote $g = \frac{d^k}{dt^k} f$ when there is no confusion. We define the Sobolev space
\begin{align*}
    H^k(I, \mathbb{R}^n) :=& \{f: I \to \mathbb{R}^n \ \vert \ f \text{ is } C^{k-1} \\&\text{ and has } k^{\text{th}} \text{ weak derivative in } L^2(I, \mathbb{R}^n) \}.
\end{align*}
It is well-known that $H^k(I, \mathbb{R}^n)$ becomes a Hilbert space when equipped with the inner product 
\begin{equation*}
    \left<f, g\right>_{H^k} = \sum_{j=0}^k\int_I \left(\frac{d^j}{dt^j}f(t) \cdot \frac{d^j}{dt^j}g(t)\right)dt = \sum_{j=0}^k \left< \frac{d^j}{dt^j}f, \frac{d^j}{dt^j}g\right>_{L^2}.
\end{equation*}
The inner product $\left< \cdot, \cdot \right>_{H^k}$ induces the norm 
\begin{equation*}
    \left\|f \right\|_{H^k} = \left[\sum_{j=0}^k\int_I \left\|\frac{d^j}{dt^j}f(t)\right\|_{\mathbb{R}^n}^2 dt\right]^{1/2},
\end{equation*}
where $\|\cdot\|_{\mathbb{R}^n}$ is the Euclidean norm on $\mathbb{R}^n$. An alternative way to construct $H^k(I, \mathbb{R}^n)$ is to define it as the completion of the space of smooth functions $C^{\infty}(I, \mathbb{R}^n)$ with respect to the norm $\|\cdot\|_{H^k}$. It can be shown that the two characterizations of $H^k(I, \mathbb{R}^n)$ are indeed equivalent. 

Let $Q$ be a finite-dimensional smooth manifold. We define $H^k(I, Q)$ as the set of curves $q: I \to H$ such that for all coordinate charts $(U, \varphi)$ on $Q$ such that $q(I') \subset U$ for some $I' \subset I$, the chart representation $\varphi \circ q: I' \to \mathbb{R}^{\dim(Q)}$ is of Sobolev class $H^k(I', \mathbb{R}^{\dim(Q)}).$ It can be seen that $H^k(I, Q)$ is an infinite-dimensional smooth manifold modelled on $H^k(I, \mathbb{R}^{\dim(Q)})$ (\cite{Jost}). It should be noted that $H^k(I, Q)$ is not in general a Hilbert space, as it not generally a vector space and thus has no well-defined inner product structure. However, the tangent space $T_q H^k(I, Q)$ of Sobolev class $H^k$ vector fields along a curve $q$ on $Q$ may be identified with $H^k(I, \mathbb{R}R^{\dim(Q)})$ and hence is a Hilbert space.

Suppose that $(Q, g)$ is a finite-dimensional complete Riemannian manifold and $q \in H^k(I, Q)$. We equip the vector space $\Gamma(q)$ of smooth vector fields along $q$ with the norm
\begin{equation}\label{eq: norm Riemannian sobolev space}
    \|X\|_{H^k_g} := \left[\sum_{j=0}^k \int_I \|D_t^j X \|^2_g dt \right]^{1/2}
\end{equation}
where $D_t^j$ denotes the $j^{\text{th}}$ covariant derivative along $q$ with respect to the Levi-Civita connection (by convention, we take $D_t^0 X = X$), and $\|\cdot\|_g$ denotes the norm induced by the Riemannian metric $g$. Denote the completion of $\Gamma(q)$ under $\|\cdot\|_{H^k_g}$ by $H^k_g(q)$. Consider an orthonormal basis of parallel vector fields $\{\xi_i\} \in \Gamma(q)$ with respect to $g$. It follows that if $X = X^i \xi_i$, then $D_t^j X = X^{i (j)} \xi_i$ for all $j \in \N$. If we let $A \cdot B := A^i B^i$ for all $A, B \in \Gamma(q)$, where $A = A^i \xi_i$ and $B = B^i \xi_i$, then we have
\begin{equation}
    \|X\|_{H^k_g} := \left[\sum_{j=0}^k \int_I \left(X^{(j)} \cdot X^{(j)} \right) dt \right]^{1/2},
\end{equation}
from which it follows that $H^k_g(I, Q)$ can be identified with $H^k(I, \mathbb{R}^n).$ Hence any complete Riemannian metric on $Q$ induces a Hilbert structure associated to $\Gamma(q)$ that coincides with that of $T_q H^k(I, Q)$. Moreover, it follows that the inner product

\begin{equation}\label{eq: Riemannian metric sobolev space}
    \left<X, Y\right>_{H^k_g} := \sum_{j=0}^k \int_I g\left(D_t^j X, D_t^j Y\right) dt,
\end{equation}

\noindent varies smoothly across the tangent spaces, and hence is a Riemannian metric on $H^k(I, Q)$\footnote{For any Hilbert manifold $M$ modelled on a Hilbert space $H$, any Riemannian structure placed on $M$ is locally equivalent to the Hilbert structure on $H$. In particular, for any local coordinate chart $(\varphi, U)$ on $M$ and point $x \in U$, there exists a unique bounded, positive-definite, self-adjoint operator $g(x)$ on $H$ such that $\left<X, Y\right>_M = \left<g(x)\varphi_\ast(X), \varphi_\ast(Y)\right>_H$ for all $X, Y \in T_x M$. Moreover, the map $x \mapsto g(x)$ is smooth on $U$.}.

In applications, especially when ultimately interested in solutions to some order $2k$ ODE, it is often the case that you wish to consider curves of some specified regularity which satisfy a set of boundary values. For example, curves of Sobolev class $H^{2k}$ regularity whose first $k$ (covariant) derivatives (including positions as $k=0$) satisfy some specified boundary conditions. For that reason, we define the path space:
\begin{align}
    \Omega^{(2k)} =& \{q \in H^{2k}([a, b], Q)\ \vert \ q(a) = q_a, \ q(b) = q_b, \\ \label{eq:Riemannian polynomial path space} & \ D_t^j \dot{q}(a) = \xi^{j}_a\text{ and } D_t^j \dot{q}(b) = \xi^{j}_b \text{ for } j=0,1,\dots, k-1 \}\nonumber
\end{align}
where $q_a, q_b \in Q$ and $\xi^{j}_a \in T_{q_a} Q$ and $\xi^j_b \in T_{q_b} Q$ for all $1 \le j \le k-1.$ 

It is easy to see that $\Omega^{(2k)}$ is the inverse image of $\left((q_a, q_b), (\xi^0_a, \xi^0_b), \dots (\xi^{k-1}_a, \xi^{k-1}_b) \right)$ under the map\newline $F: H^{2k}([a, b], Q) \to TQ^k$ given by \newline $F(q) = \left((q(a), q(b)), (\dot{q}(a), \dot{q}(b)), \dots, (D^{k-1}\dot{q}(a), D^{k-1} \dot{q}(b)) \right)$. Moreover, it can be shown that $F$ is a smooth submersion, from which it follows by the implicit function theorem that $\Omega^{(2k)}$ is a closed submanifold of $H^{2k}([0, 1], Q),$ and hence inherits its Hilbert structure. The tangent space $T_q \Omega^{(2k)}$ can be indentified with the space $X \in \mathring{H}^{2k}_g(q)$ of vector fields in $H^{2k}_g(q)$ which vanish at the endpoints along with their first $k$ covariant derivatives. Hence we may equip $\Omega^{(2k)}$ with the Riemannian structure \eqref{eq: Riemannian metric sobolev space}. We also consider the special case 
\begin{equation}\label{eq: path space geodesics}
    \Omega^{(1)} = \{q \in H^{1}([a, b], Q) \ \vert \ q(a) = q_a, \ q(b) = q_b\}
\end{equation}
which is itself a closed submanifold of $H^{1}([a, b], Q)$, and is of particular importance for geodesics. We will occasionally use the notation $\Omega^{(1), [a, b]}_{q_a, q_b}(Q)$ (and similar for higher-order path spaces) when it is necessary to refer to the boundary conditions, underlying manifold, and interval of integration.

\subsection{Riemannian geometry on Lie Groups}\label{sec: background_Lie}

Let $G$ be a Lie group with Lie algebra $\g := T_{e} G$, where $e$ is the identity element of $G$. The left-translation map $L: G \times G \to G$ provides a group action of $G$ on itself under the relation $L_{g}h := gh$ for all $g, h \in G$. Given any inner-product $\left< \cdot, \cdot \right>_{\g}$ on $\g$, left-translation provides us with a Riemannian metric $\left< \cdot, \cdot \right>$ on $G$ via the relation:
\begin{align*}
    \left< X_g, Y_g \right> := \left< g^{-1} X_g, g^{-1} Y_g \right>_{\g},
\end{align*}
for all $g \in G, X_g, Y_g \in T_g G$. Such a Riemannian metric is called \textit{left-invariant}, and it follows immediately that there is a one-to-one correspondence between left-invariant Riemannian metrics on $G$ and inner products on $\g$, and that $L_g: G \to G$ is an isometry for all $g \in G$ by construction. Any Lie group equipped with a left-invariant metric is complete as a Riemannian manifold. In the remainder of the section, we assume that $G$ is equipped with a left-invariant Riemannian metric.

In the following $L_{g^{\ast}}$ stands for the push-forward of $L_g$, which is well-defined because $L_g: G \to G$ is a diffeomorphism for all $g \in G$. We call a vector field $X$ on $G$ \textit{left-invariant} if $L_{g\ast} X = X$ for all $g \in G$, and we denote the set of all left-invariant vector fields on $G$ by $\mathfrak{X}_L(G)$. It is well-known that the map $\phi: \g \to \mathfrak{X}_L(G)$ defined by $\phi(\xi)(g) = L_{g\ast} \xi$ for all $\xi \in \g, g \in G$ is an isomorphism between vector spaces. This isomorphism allows us to construct an operator $\nabla^{\g}: \g \times \g \to \g$ defined by:
\begin{align}
    \nabla^{\g}_{\xi} \eta := \nabla_{\phi(\xi)} \phi(\eta)(e),\label{g-connection}
\end{align}
for all $\xi, \eta \in \g$, where $\nabla$ is the Levi-Civita connection on $G$ corresponding to the left-invariant Riemannian metric $\left< \cdot, \cdot \right>$. Although $\nabla^{\g}$ is not a connection, we shall refer to it as the \textit{Riemannian $\g$-connection} corresponding to $\nabla$ because of the similar properties that it satisfies:
\begin{lemma}\label{lemma: covg_prop}
$\nabla^\g: \g \times \g \to \g$ is $\R$-bilinear, and for all $\xi, \eta, \sigma \in \g$, the following relations hold:
\begin{equation*}
 \hbox{(1) }\nabla_{\xi}^{\g} \eta - \nabla_{\eta}^{\g} \xi = \left[ \xi, \eta \right]_{\g},\, \hbox{(2) } \left< \nabla_{\sigma}^{\g} \xi, \eta \right> + \left<  \xi, \nabla_{\sigma}^{\g}\eta \right> = 0.
\end{equation*}
\end{lemma}

\begin{remark}\label{remark: covg_operator}
We may consider the Riemannian $\g$-connection as an operator $\nabla^\g: C^{\infty}([a, b], \g)\times C^{\infty}([a, b], \g) \to C^{\infty}([a, b], \g)$ in a natural way,  namely, if $\xi, \eta \in C^{\infty}([a, b], \g)$, we can write $(\nabla^\g_{\xi} \eta)(t) := \nabla^\g_{\xi(t)}\eta(t)$ for all $t \in [a, b]$. With this notation, Lemma \ref{lemma: covg_prop} works identically if we replace $\xi, \eta, \sigma \in \g$ with $\xi, \eta, \sigma \in C^{\infty}([a, b], \g)$.
\end{remark}

Given a basis $\{A_i\}$ of $\g$, we may write any vector field $X$ on $G$ as $X = X^i \phi(A_i)$, where $X^i: G \to \R$, where we have adopted the Einstein sum convention. If $X$ is a vector field along some smooth curve $g: [a, b] \to G$, then we may equivalently write $X = X^i g A_i$, where now $X^i: [a, b] \to \R$ and $g A_i =: L_g A_i$. We denote $\dot{X} = \dot{X}^i A_i$, which may be written in a coordinate-free fashion via $\dot{X}(t) = \frac{d}{dt}\left(L_{g(t)^{-1 \ast}}X(t) \right)$. We now wish to understand how the Levi-Civita connection $\nabla$ along a curve is related to the Riemannian $\g$-connection $\nabla^\g$. This relation is summarized in the following result \cite{goodman2022reduction}.

\begin{lemma}\label{lemma: cov-to-covg}
Consider a Lie group $G$ with Lie algebra $\g$ and left-invariant Levi-Civita connection $\nabla$. Let $g: [a,b] \to G$ be a smooth curve and $X$ a smooth vector field along $g$. Then the following relation holds for all $t \in [a, b]$:
\begin{align}
    D_t X(t) = g(t)\left(\dot{X}(t) + \nabla_{\xi}^\g \eta(t) \right).\label{eqs: Cov-to-covg}
\end{align}
\end{lemma}
\begin{lemma}\label{lemma: covg-decomp}
The Riemannian $\g$-connection satisfies:
\begin{align}
    \nabla_{\xi}^\g \eta = \frac12\left( [\xi, \eta]_\g - \ad^\dagger_{\xi} \eta - \ad^\dagger_{\eta} \xi\right),\label{eq: covg_decomposition}
\end{align} 
for all $\xi, \eta \in \g$.
\end{lemma}
\section{The Collision Avoidance Task}\label{subsection: safety collision avoidance}

We now switch our attention to multi-agent systems and the collision avoidance task. Consider a set $\mathcal{V}$ consisting of $s\geq 2$ agents on $Q$, {a complete and connected Riemannian manifold}. The configuration of each agent at any given time is determined by the element $q_i(t)\in Q$, $i=1,\ldots,s$. The neighboring relationships are described by an undirected time-invariant graph $\mathbb{G} = (\mathcal{V}, \mathcal{E})$ with edge set $\mathcal{E}\subseteq\mathcal{V}\times\mathcal{V}$. The set of neighbors $\mathcal{N}_i$ for the agent $i\in\mathcal{V}$ is given by $\mathcal{N}_i=\{j\in\mathcal{V}:(i,j)\in\mathcal{E}\}$. An agent $i\in\mathcal{V}$ can measure its Riemannian distance from other agents in the subset $\mathcal{N}_i \subseteq \mathcal{V}$. {The assumptions that $Q$ is complete and connected and $\mathbb{G}$ is undirected and time-invariant will remain for the remainder of the paper}.

For $i=1,...,s$, and points $\xi_i, \eta_i \in TQ$, consider the sets $\Omega_i := \Omega_{\xi_i, \eta_i}^{(2), [a, b]}$ and the functional $J_{\coll}$ on $\Omega := \Omega_1\times\cdot\cdot\cdot\times\Omega_s$ defined by:
\begin{align}\label{J_collision_avoidance}
J_{\coll}(q_1,q_2,\ldots,q_s)= \frac12\sum\limits_{i=1}^s \int_a^b \Big{(} \Big{|}\Big{|}&\frac{D \dot{q}_i}{dt}(t)\Big{|}\Big{|}^2 \\&+ \sum\limits_{(i, j) \in \mathcal{E}}V_{ij}(q_i(t),q_j(t))\Big{)}dt.\nonumber
\end{align}
where $V_{ij}:Q\times Q\to\mathbb{R}$ is a smooth non-negative function called an \textit{artificial potential} satisfying the symmetry relations $V_{ij} = V_{ji}$ and $V_{ij}(p, x) = V_{ij}(x, p)$ for all $(i,j)\in\mathcal{E}$ and $(p,x) \in Q\times Q$. 

The sets $Q^s := Q \times Q \times \dots Q$ ($s$ times) and $\Omega$ are (complete) Riemannian manifolds when equipped with the respective product metrics. Moreover, $\Omega$ is an infinite-dimensional Hilbert manifold with model space $(H^2([a, b], \mathbb{R}^n))^s$, which may be identified with the path space $\Omega^{(2)}$ of $Q^s$. With this identification, it is clear that $J_{\coll}$ may be identified with $J: \Omega^{(2)} \to \R$, where $V: Q^s \to \R$ is given by $V(q) = V(q^1, \dots, q^s) := \sum\limits_{j\in \mathcal{N}_i}V_{ij}(q_i(t),q_j(t))$. It follows that the analysis done in \cite{goodman2021obstacle} also applies to $J_{\coll}$ and its minimizers. In particular, the necessary conditions for optimality take the following form:

\begin{proposition}\label{prop: necessary conditions collision avoidance}
{A curve $q = (q_1,...,q_s) \in \Omega$ is a critical point of $J_{\coll}$ if and only if for each $i \in \mathcal{V}$, $q_i \in \Omega_i$ is smooth and for all $t \in [0, T]$ satisfies}
\begin{equation}\label{eq: necessary conditions collision avoidance}
    D^3_t \dot{q}_i + R\Big{(}D_t \dot{q}_i,\dot{q}_i\Big{)}\dot{q}_i=-\sum\limits_{j\in \mathcal{N}_i} \hbox{\grad}_1 \, V_{ij}(q_i(t),q_j(t)),
\end{equation}
where $R$ denotes the curvature endomorphism on $Q$ and $\grad_1 V_{ij}$ denotes the gradient vector field of $V_{ij}$ with respect to the first argument.
\end{proposition}
\section{Reduction by Symmetry on Lie Groups}\label{section: Reduction by symmetry Lie Groups}

In this section, we reduce the necessary conditions \eqref{eq: necessary conditions collision avoidance} by symmetry on a Lie group $G$ equipped with a left-invariant Riemannian metric. This process amounts to left-translating the necessary conditions on $G$ to some equivalent set of equations on the Lie algebra $\g$, together with a reconstruction equation. In Section \ref{sec: reduction cubics bi-invariant metrics}, we will additionally consider the special case that $G$ admits a bi-invariant metric, and show how the gradient vector field of the artificial potential can be calculated explicitly in the collision avoidance problem. Reduced collision avoidance extremals for rigid body motions in $\SO(3)$ are considered as an example. 

\subsection{Reduction of Necessary Conditions}\label{subsection: reduction Lie groups left-invariant metrics}

We wish to obtain Euler-Poincar\'e equations corresponding to \eqref{eq: necessary conditions collision avoidance} under the following assumption:
\begin{quote}
\begin{description}
    \item[\textbf{G1:}] $Q = G$ is a connected Lie group endowed with a left-invariant Riemannian metric and corresponding Levi-Civita connection $\nabla$.
\end{description}
\end{quote}

\noindent To do this, we first must understand the forms that $D_t^3 \dot{g}$ and $R(D_t \dot{g}, \dot{g})\dot{g}$ take when left-translated to curves in the Lie algebra $\g$. This is summarized in the following lemma:

{\begin{lemma}\label{lemma: cubic_red}
Let $g: [a, b] \to G$ be a smooth curve and set $\xi^{(0)} := g^{-1}\dot{g}$. Recursively define $\xi^{(i)} = \dot{\xi}^{(i-1)} + \nabla^\g_{\xi^{(0)}} \xi^{(i-1)}$ for $i = 1,2$. Then,
\begin{align}
    D^3_t \dot{g} &= g\Big{(} \dot{\xi}^{(2)} + \nabla^\g_{\xi^{(0)}} \xi^{(2)} \Big{)},\label{covTTTT} \\
    R\left(D_t \dot{g}, \dot{g}\right)\dot{g} &= g R\big{(}\xi^{(1)}, \xi^{(0)} \big{)}\xi^{(0)}.\label{R_g}
\end{align}
\end{lemma}
\begin{proof}
It's clear from Lemma \ref{lemma: cov-to-covg} and Lemma \ref{lemma: covg-decomp} that $\xi^{(i)} = g^{-1}D_t^i \dot{g}$ for $i = 0,1,2.$ One more application of Lemma \ref{lemma: cov-to-covg} to $D_t^2 \dot{g} = g \xi^{(2)}$ yields equation \eqref{covTTTT}.
Equation \eqref{R_g} follows immediately by the fact that $R$ is a left-invariant tensor field and observing that $R\left(D_t \dot{g}, \dot{g}\right)\dot{g} = R\big{(}g\xi^{(1)}, g\xi^{(0)} \big{)}g\xi^{(0)} = g R\big{(}\xi^{(1)}, \xi^{(0)} \big{)}\xi^{(0)}$.
\end{proof}}
%\nabla_S \nabla_T T &= \Gamma\left( (\dot{\Tg})' + \nabla^\g_{\Sg} \dot{\Tg} + \nabla^\g_{\Tg'} \Tg + \nabla^\g_{\Tg} \Tg' + \nabla^\g_{\Sg} \nabla_{\Tg}^\g \Tg \right) \\
    %\nabla_T \nabla_T T &= \Gamma\left( \ddot{\Tg} + 2 \nabla^\g_{\Tg} \dot{\Tg} + \nabla^\g_{\dot{\Tg}} \Tg + \nabla^\g_{\Tg} \nabla_{\Tg}^\g \Tg \right) \\
    %\nabla_T \nabla_S T &= \Gamma\left( (\dot{\Tg})' + \nabla^\g_{\Tg} \Tg' + \nabla^\g_{\dot{\Sg}} \Tg + \nabla^\g_{\Sg} \dot{\Tg} + \nabla^\g_{\Tg} \nabla_{\Sg}^\g \Tg \right) \\

The quantities calculated in Lemma \ref{lemma: cubic_red} may be substituted directly into equation \eqref{eq: necessary conditions collision avoidance}. If $g$ were a Riemannian cubic polynomial (i.e., in the case that $V \equiv 0$), we would immediately obtain reduced equations on the Lie algebra $\g$. However, we still must handle the artificial potential. We may left-translate the gradient potential directly to the Lie algebra. That is, we write $\grad V(g) = L_{g^\ast} \circ L_{g^{-1 \ast}} \grad V(g)$, and substitute this directly into \eqref{eq: necessary conditions collision avoidance} along with equations \eqref{covTTTT} and \eqref{R_g} to obtain the reduced equations. 

%We now switch our attention to the collision avoidance task. In particular, we define an artificial potential $V: G \to \R$ by $V(g) = f(d^2(g, g_0))$ for all $g \in G$, where $f: \R \to \R$ is a smooth non-negative function, $d^2: G \times G \to \R$ is the square of the Riemannian distance on $G$, and $g_0 \in G$ is the obstacle that we wish to avoid. 

First, we show that the Riemannian distance $d$ is invariant under left-translation, which follows immediately by the left-invariance of the metric:

\begin{lemma}\label{lemma: distance-left-inv}
$d(gq, gp) = d(q, p)$ for all $g,q, p \in G$.
\end{lemma}
\begin{proof}
Since $G$ is complete as a Riemannian manifold, there exists a geodesic $\gamma: [0, 1] \to G$ which minimizes the length functional $\displaystyle{L(c) = \int_0^1 \|\dot{c}(t)\| dt}$ among all smooth curves $c:[0, 1] \to G$ satisfying $\ c(0) = p, \ c(1) = q$. Moreover, we have $d(p, q) = L(\gamma)$ by equation \eqref{eq: distance-length-geodesics}. By left-invariance of the metric, we then have that $d(p, q) = L(\gamma) = L(g\gamma) \ge d(gp, gq)$, since in particular $g \gamma$ is a smooth curve such that $g\gamma(0) = gp, \ g\gamma(1) = gq$. On the other hand, there exists some geodesic $\gamma^\ast$ such that $L(\gamma^\ast) = d(gp, gq)$, and so $d(gp, gq) = L(\gamma^\ast) = L(g^{-1}\gamma^\ast) \ge d(p, q)$. It follows that $d(p, q) = d(gp, gq)$.
\end{proof}

For the purposes of collision avoidance, the particular sub-potentials $V_{jk}$ will take the form $V_{jk}(g_j, g_k) = f(d^2(g_j, g_k)),$ where $d^2: G \times G \to \R$ is the square of the Riemannian distance on $G$. That is, we have $V_{jk}(hg_j, hg_k) = V_{jk}(g_j, g_k)$ for all $h \in G$, $(j,k) \in \mathcal{E}$, from which it is clear that $V$ is left-invariant with respect to left-translation on $G^s$. 

Observe that $g_j^{-1} \grad_1 V_{jk}(g_j, g_k) = \grad_1 V_{jk}(e, g_j^{-1} g_k)$. This motivates the definition $h_{jk} = g_j^{-1} g_k$, from which we find that $\dot{h}_{jk} = -g_j^{-1} \dot{g}_j g_j^{-1} g_k + g_j^{-1} \dot{g}_k = -\xi_j h_{jk} + h_{jk}\xi_k$, where $\xi_i := g_i^{-1} \dot{g}_i$ for all $i = 1, \dots, s$. This leads to the following result:

%\todo{Define $\text{Stab}(g_0)$ and $\grad_1$}

\begin{proposition}\label{prop: reduction_left_inv}
Suppose that $Q = G^s$, where $G$ satisfies assumption \textbf{G1}. Then $g = (g_1, \dots, g_s) \in C^{\infty}([a,b], G^s)$ satisfies \eqref{eq: necessary conditions collision avoidance} if and only if $\xi^{(0)}_j := g^{-1}_j \dot{g}_j$ and $h_{jk} := g^{-1}_j g_k$ solve:
\begin{align}
    \dot{h}_{jk} &= -\xi^{(0)}_j h_{jk} + h_{jk}\xi_k^{(0)}, \label{eq: reduction modified cubic parameter coll} \\
    \dot{\xi}^{(i)}_j &= \xi^{(i+1)}_j - \nabla^\g_{\xi^{(0)}_j} \xi^{(i)}_j, \label{eq: modified cubic reduction recusion coll}\\
   \dot{\xi}^{(2)}_j + \nabla_{\xi^{(0)}_j}^\g \xi^{(2)}_j + R\big{(}\xi^{(1)}_j, \xi^{(0)}_j \big{)}\xi^{(0)}_j &= -\sum_{r \in \mathcal{N}_j} \grad_1 V_{jr}(e, h_{jr}), \label{eq: reduced modified cubic coll}
\end{align}
for $i=0,1$, and for all $j = 1, \dots, s$ and $k \in \mathcal{N}_j$.
\end{proposition}

\begin{remark}
Proposition \ref{prop: reduction_left_inv} can be considered as a special case of Euler-Poincar\'e reduction for second order Lagrangians (that is, Lagrangians defined on the second order tangent bundle $T^2 G$). This was studied on Lie groups in \cite{colombo2014higher}, where the corresponding higher order Euler-Poincar\'e equations were obtained. Using the Riemannian formalism, we bypass the necessity to work with higher-order tangent bundles, and obtain equations evolving Lie algebra $\g$ rather than its dual $\g^\ast$. Also Proposition \ref{prop: reduction_left_inv}  can be seen as the second-order extension of the collision avoidance problem on Lie groups considered in \cite{colombo2020symmetry}.
\end{remark}
\section{Reduction on Lie Groups with Bi-invariant Metrics}
\subsection{Bi-invariant metrics}\label{subsection: bi-invariant metrics}

Now we wish to discuss another important class of Riemannian metrics on a Lie group, the so-called \textit{bi-invariant} (or \textit{$\Ad$-invariant}) metrics. These are the Riemannian metrics $\left< \cdot, \cdot\right>$ on $G$ which are both left- and right-invariant. Unlike left- and right-invariant metrics, not every Lie group $G$ admits a bi-invariant metric. The following result from \cite{LieGroupRiem} provides necessary and sufficient conditions for the existence of a bi-invariant metric.

\begin{lemma}\label{lemma: existence of bi-invariant metrics}
A connected Lie group admits a bi-invariant metric if and only if it is isomorphic to the Cartesian product of a compact Lie group and a finite-dimensional vector space. Moreover, such a metric is unique up to scalar multiplication.
\end{lemma}

Despite this limitation, many important examples of Lie groups satisfy the conditions. In particular, $\SO(3)$ is a compact Lie group, and $\mathbb{R}^3$ is a finite-dimensional vector space. Bi-invariant metrics have many nice properties that greatly simplify calculations in practice. First, it is clear that for all $g \in G$,  $X, Y \in T_gG$, we have $\left<X, Y\right> = \left<gXg^{-1}, gYg^{-1}\right> = \left<\Ad_g X, \Ad_g Y\right>$ where $\Ad:G\times\mathfrak{g}\to\mathfrak{g}$ is the adjoint operator (this is why such metrics are also called $\Ad$-invariant). Let $\xi, \eta, \sigma \in \g$. Then, $\left<\eta, \sigma \right> = \left<\Ad_{\Exp(t\xi)} \eta, \Ad_{\Exp(t\xi)} \sigma\right>$. Differentiating at $t =0$, we see that $0 = \left<\ad_{\xi}\eta, \sigma\right> + \left<\eta, \ad_{\xi}\sigma\right>$, which implies that $\left<\ad^\dagger_\xi \eta, \sigma\right> = \left<-\ad_\xi \eta, \sigma\right>$. Hence, $\ad^\dagger_\xi \eta = -\ad_\xi \eta = [\eta, \xi]$ for all $\xi, \eta \in \g$. 

\begin{lemma}\label{lemma: g-connection bi-inv metric}
Consider a Lie group $G$ equipped with a bi-invariant metric. Let $\nabla$ be the Levi-Civita connection and $\nabla^\g$ be the corresponding Riemannian $\g$-connection. Then:
\begin{enumerate}
    \item $\nabla^\g_\xi \eta = \frac12 [\xi, \eta]$,
    \item $R(\xi, \eta)\sigma = \frac14[[\xi, \eta],\sigma]$,
\end{enumerate}
for all $\xi, \eta, \sigma \in \g$.
\end{lemma}

%The \textit{Lie exponential map} $\Exp: \g \to G$ is defined for all $\xi \in \g$ by $\Exp(\xi) = \gamma(1)$, where $\gamma$ is the unique solution to $\dot{\gamma} = \gamma \xi$ with $\gamma(0) = e$. %That is, the one-parameter subgroups of $G$ are exactly the geodesics through the identity.
%The Lie exponential map $\Exp: \g \to G$ agrees with the restriction of the Riemannian exponential map $\exp: TG \to G$ to the Lie algebra, where $\exp$ is taken with respect to the Levi-Civita connection of a bi-invariant metric. In particular, this implies that geodesics through the identity are just one-parameter subgroups (and vice verse).

%\begin{theorem}\label{thm: exponential maps equivalence bi-inv}
%Consider a Lie group $G$ equipped with a bi-invariant metric. For all $\xi \in \g$, we have $\Exp(\xi) = \exp_e(\xi)$, where $\Exp$ is the Lie exponential map, and $\exp_e: \g \to G$ is the Riemannian exponential map at the identity $e \in G$ with respect to the Levi-Civita connection. 
%\end{theorem}

%Theorem \ref{thm: exponential maps equivalence bi-inv} turns out to be very useful in many applications on compact matrix Lie groups, due to the simple expression for the Lie exponential map. In cases where we are able to calculate the Logarithmic map $\Log: G \to \g$ directly (on some neighborhood of the identity element $e \in G),$ we may also calculate the Riemannian distance explicitly. In particular, we find that $d(e, g) = \|\Log(g)\|$ wherever it is defined. 

\subsection{Reduction on Lie Groups with Bi-invariant Metrics}\label{sec: reduction cubics bi-invariant metrics}

%In this section, we study the special case where $G$ admits a bi-invariant metric. By Lemma \ref{lemma: existence of bi-invariant metrics}, this occurs in any application where $G$ is a compact Lie group (such as $\SO(3)$). In such a case

Necessary conditions for optimality can be simplified dramatically in the case that the metric $\left<\cdot, \cdot\right>$ is bi-invariant. In particular, we view equipping $G$ with such a bi-invariant metric as a strengthening of assumption \textbf{G1}:
\begin{quote}
\noindent \textbf{G2} \textbf{:} $G$ is a connected Lie group equipped with a bi-invariant Riemannian metric and corresponding Levi-Civita connection $\nabla$.
\end{quote}

\noindent Using Lemma \ref{lemma: g-connection bi-inv metric}, we obtain the following corollary to Proposition \ref{prop: reduction_left_inv}.
\begin{corollary}\label{cor: reduction_bi_inv}
Suppose that $G^s$ satisfies assumption \textbf{G2}. Then $g_j \in \Omega$ satisfies that $\xi^{(0)}_j := g^{-1}_j \dot{g}_j$ and $h_{jk} := g^{-1}_j g_k$ solve:
\begin{align}
    \dot{h}_{jk} &= -\xi^{(0)}_j h_{jk} + h_{jk}\xi_k^{(0)} \label{eqqalpha_bi} \\
    \dddot{\xi_j} + \big{[}\xi_j, \ddot{\xi}_j\big{]} &= -\sum_{r \in \mathcal{N}_j} \grad_1 V_{jr}(e, h_{jr}). \label{eqq3}
\end{align}
%together with the boundary conditions $\xi_j(a) = (g_a^j)^{-1} v^j_a, \ \xi_j(b) = g_b^{-1} v_b$, and $h(a) = g_a^{-1} g_0, \ h(b) = g_b^{-1} g_0$.
\end{corollary}
\begin{proof}
From Proposition \ref{prop: reduction_left_inv}, it follows that if we take $\xi^{(0)} = \xi$, then $g \in \Omega$ solves \eqref{eq: necessary conditions collision avoidance} if and only if:
\begin{align*}
    \dot{h}_{jk} &= -\xi^{(0)}_j h_{jk} + h_{jk}\xi_k^{(0)}, \\
   \hbox{for  $i=0,1,$ }\qquad \dot{\xi}^{(i)}_j &= \xi^{(i+1)}_j - \frac12\big{[}\xi_j, \xi_j^{(i)}\big{]}\\
   \dot{\xi_j}^{(2)} + \frac12\big{[}\xi_j, \xi_j^{(2)}\big{]} - \frac14 \big{[}\big{[}\xi_j^{(1)}, \xi_j \big{]}, \xi_j \big{]} &=  -\sum_{r \in \mathcal{N}_j} \grad_1 V_{jr}(e, h_{jr}).
\end{align*}
Now observe that, from Lemmas \ref{lemma: cubic_red} and \ref{lemma: g-connection bi-inv metric}:
\begin{align*}
    \dot{\xi_j} &= \xi_j^{(1)} - \frac12\big{[} \xi_j, \xi_j \big{]} = \xi_j^{(1)}, \\
    \ddot{\xi_j} &= \dot{\xi_j}^{(1)} = \xi_j^{(2)} - \frac12\big{[}\xi_j, \dot{\xi}_j \big{]}, \\
    \dddot{\xi_j} &= \dot{\xi}_j^{(2)} - \frac12 \big{[}\xi_j, \ddot{\xi}_j\big{]}.
\end{align*}
Substituting these into the necessary conditions, we obtain \eqref{eqqalpha_bi} and \eqref{eqq3}.
\end{proof}
%\begin{proof}
%Observe that, as a consequence of equation \eqref{cov_bi}, the following two relations hold for all $\xi, \eta \in \g$:
%\begin{align*}
%    \nabla^\g_{\xi} \eta &= -\nabla^\g_{\eta} \xi, \\
%    \nabla^\g_{\xi} \xi &= 0.
%\end{align*}
%Moreover, the curvature endomorphism relates to the Riemannian $g$-connection as follows:
%\begin{align*}
%    R(\xi, \eta)\sigma &= -\frac14 \big{[}\big{[}\xi, \eta \big{]}, \sigma \big{]}= -\frac14 \big{[}\sigma, \big{[}\eta, \xi \big{]} \big{]}= -\frac12 \big{[}\sigma, \nabla^\g_{\eta} \xi \big{]} = -\nabla^\g_{\sigma} \nabla^\g_{\eta} \xi.
%\end{align*}
%Applying these identities to \eqref{eq: reduced modified cubic} term-by-term yields \eqref{eqq3}.
%\end{proof}

\bibliography{ifacconf}             % bib file to produce the bibliography
\end{document}